\documentclass{amsart}
\usepackage{amsmath,amscd,amsthm}
\usepackage{tikz}
\usepackage{tikz-cd}
\usepackage{color}

\definecolor{dark_purple}{rgb}{0.4, 0.0, 0.4}
\definecolor{dark_green}{rgb}{0.0, 0.7, 0.0}

\usepackage{varwidth}
\usepackage{caption}
\usepackage{graphicx}

\newtheorem{theorem}{Theorem}[section]
\theoremstyle{definition}
\newtheorem{example}[theorem] {Example}

\newtheorem{proposition}[theorem] {Proposition}
\newtheorem{remark}[theorem] {Remark}

\newtheorem{lemma}[theorem] {Lemma}

\newtheoremstyle{Theorem-intro}{}{}{\itshape}{}{\bfseries}{.}{ }{\thmname{#1}\thmnote{ \bfseries #3}}
\theoremstyle{Theorem-intro}
\newtheorem{thm-intro}{Theorem}

\newtheoremstyle{Proposition-intro}{}{}{\itshape}{}{\bfseries}{.}{ }{\thmname{#1}\thmnote{ \bfseries #3}}
\theoremstyle{Proposition-intro}
\newtheorem{prop-intro}{Proposition}

\newtheoremstyle{Corollary-intro}{}{}{\itshape}{}{\bfseries}{.}{ }{\thmname{#1}\thmnote{ \bfseries #3}}
\theoremstyle{Corollary-intro}
\newtheorem{cor-intro}{Corollary}

\newtheoremstyle{Lemma-intro}{}{}{\itshape}{}{\bfseries}{.}{ }{\thmname{#1}\thmnote{ \bfseries #3}}
\theoremstyle{Lemma-intro}
\newtheorem{lem-intro}{Lemma}

\def \g{\mathfrak{g}}
\def \h{\mathfrak{h}}
\def \m{\mathfrak{m}}
\def \k{\mathfrak{k}}
\def \p{\mathfrak{p}}
\def \t{\mathfrak{t}}
\def \f{\mathfrak{f}}
\def \l{\mathfrak{l}}
\def \u{\mathfrak{u}}
\def \q{\mathfrak{q}}

\def \sl{\mathfrak{sl}}

\def \spin{\mathfrak{spin}}
\def \Spin{\mathrm{Spin}}
\def \U{\mathrm{U}}
\def \SU{\mathrm{SU}}
\def \SO{\mathrm{SO}}

\def \C{\mathcal{C}}
\def \D{\mathcal{D}}
\def \V{\mathcal{V}}
\def \H{\mathcal{H}}
\def \N{\mathcal{N}}
\def \O{\mathcal{O}}
\def \d{\mathrm{d}}
\def \i{\mathrm{i}}

\def \zn{\mathbb{Z}}
\def \rn{\mathbb{R}}
\def \cn{\mathbb{C}}
\def \hn{\mathbb{H}}
\def \Oc{\mathbb{O}}
\def \P{\mathbb{P}}

\DeclareMathOperator{\tr}{tr}

\DeclareMathOperator{\Ad}{Ad}
\DeclareMathOperator{\ad}{ad}
\DeclareMathOperator{\im}{Im}
\DeclareMathOperator{\di}{dim}

\DeclareMathOperator{\spa}{span}
\allowdisplaybreaks

\title[]{Harmonic surfaces in the Cayley plane}

\author[N. Correia]{Nuno Correia}
\address{Centro de Matem\'{a}tica e Aplica\c{c}{\~{o}}es (CMA-UBI), Universidade da Beira Interior, 6201 -- 001
Covilh{\~{a}}, Portugal.}
\email{ncorreia@ubi.pt}
\author[R. Pacheco]{Rui Pacheco}
\address{Centro de Matem\'{a}tica e Aplica\c{c}{\~{o}}es (CMA-UBI), Universidade da Beira Interior, 6201 -- 001
Covilh{\~{a}}, Portugal.}
\email{rpacheco@ubi.pt}
\thanks{The first and the second author were partially supported by Funda\c{c}\~{a}o para a Ci\^{e}ncia e Tecnologia through the project UID/MAT/00212/2019.}
\author[M. Svensson]{Martin Svensson}
\address{Department of Mathematics and Computer Science, University of Southern Denmark, Campusvej 55, 5230 Odense M, Denmark.}
\email{svensson@imada.sdu.dk}
\thanks{}

\begin{document}

\begin{abstract}
We consider the twistor theory of nilconformal harmonic maps from a Riemann surface into the Cayley plane $\Oc P^2=F_4/\Spin(9)$.  By exhibiting this symmetric space as a submanifold of the Grassmannian of $10$-dimensional subspaces of the fundamental representation of $F_4$,  techniques and constructions similar to those used in earlier works  on twistor constructions of nilconformal harmonic maps into classical Grassmannians can also be applied in this case. The originality of our approach lies on the use of the classification of nilpotent orbits in Lie algebras as described by D. Djokovi\'{c}.
\end{abstract}

\maketitle

\section{Introduction}
In this paper we consider the twistor theory of nilconformal harmonic maps from a Riemann surface into the Cayley plane $\Oc P^2$, the $16$-dimensional exceptional Riemannian symmetric space $F_4/\Spin(9)$. Exhibiting this manifold as a submanifold of the Grassmannian of $10$-dimensional subspaces of the fundamental representation of $F_4$ allows us to use techniques and constructions similar to those used in earlier works on twistor constructions of nilconformal harmonic maps into classical Grassmannians \cite{svensson-wood2014}, as well as the exceptional Grassmannian $G_2/\SO(4)$ \cite{svensson-wood2015}.

We thus describe the canonical twistor fibrations over $\Oc P^2$ in terms of refinements of the Grassmannian structure of the base manifold. As this structure is defined in terms of the fundamental representation of $F_4$, we may use the algebraic structure of this representation to show that the nilconformality is actually implied by the conformality:

\begin{prop-intro}[\ref{prop:nilconformal}] Any weakly conformal harmonic map from a Riemann surface to $\Oc P^2$ is nilconformal with nilorder $\leq5$.
\end{prop-intro}

Following the general twistor theory for harmonic maps into compact symmetric spaces as developed in \cite{BurstallRawnsley}, we see that there are three canonical twistor fibrations over $\Oc P^2$: $T_4$, of dimension $30$; $T_3$ of dimension $40$; $T_{34}$ of dimension $42$. As the dimensions of these twistor fibrations are far higher than for example those of the exceptional Grassmannian $G_2/\SO(4)$ and too high to be handled in a straightforward fashion as in \cite{svensson-wood2015}, we have chosen an approach significantly different from that of previous works in our study of the twistor lifts of nilconformal harmonic maps, making use of the classification of nilpotent orbits in Lie algebras as described in \cite{Collingwood}. And while there are three different canonical twistor fibrations over $\Oc P^2$, only one of these, $T_4$,  a submanifold of the isotropic lines in the fundamental representation of $F_4$, is needed to describe the twistor theory of nilconformal harmonic maps. More precisely, we show the following result, where $M$ is an arbitrary Riemann surface:

\begin{thm-intro}[\ref{thm-twistor}] Any  weakly conformal harmonic map $\varphi:M\to \Oc P^2$ admits a $J_2$-holomorphic lift into $T_4$.
\end{thm-intro}

The almost complex structure $J_2$ is the classical non-integrable structure that ensures that the projection of any holomorphic map from a Riemann surface to $T_4$ to $\Oc P^2$ is harmonic. As $T_4$ is a $3$-symmetric space, $J_2$-holomorphicity for a map is equivalent to the map being \emph{primitive} \cite{BurstallPedit}. 

We next consider the special case of harmonic maps into $\Oc P^2$ of finite uniton number. Following Burstall and Guest \cite{B-G}, we show how to explicitly construct examples of $S^1$-invariant extended solutions as well as their twistor lifts. When the domain is a torus, the class of \emph{harmonic maps of finite type} (which are constructed from commuting Hamiltonian flows on finite dimensional subspaces of a loop algebra) plays a fundamental role. As a matter of fact, it was shown in \cite{BFP} that any  harmonic map of a torus $T^2$ into a compact rank one symmetric space is of finite type if it is not weakly conformal. On the other hand, any non-constant harmonic map from $T^2$ into $\cn P^n$ is either of finite type or finite uniton number but not both \cite{B,O-U,Pa1}. The corresponding statement for $\hn P^2$ is known to be false \cite{Pa2}, and consequently it  is also false for $\Oc P^2$, since $\hn P^2$  can be totally geodesically \cite{wolf} embedded in  $\Oc P^2$.

While our approach using the classification of nilpotent orbits is novel, it is nevertheless a natural approach when studying nilconformal harmonic maps into symmetric spaces or, more generally, Lie groups. Our work suggests a revisit of earlier results on this topic using this approach, as well as a path to understand the twistor theory of harmonic maps for the remaining exceptional symmetric spaces.

This paper is arranged as follows: in Section \ref{se:Cayley} we give a brief introduction to the Cayley plane $\Oc P^2$, showing how this can be viewed as a particular Grassmannian manifold related to the exceptional Jordan algebra of hermitian $3\times3$ octonionic matrices. In Section \ref{se:twistor} we construct the canonical twistor fibrations of $\Oc P^2$ and study the properties of twistor lifts of harmonic maps into $\Oc P^2$. Nilconformality is introduced in Section \ref{se:nilconformal} where we use the classification of nilptotent orbits in the Lie algebra of $F_4$ to show that any nilconformal harmonic map from a Riemann surface into $\Oc P^2$ admits a twistor lift into $T_4$. In Section \ref{se:unitons} we give some explicit examples of harmonic maps of finite uniton number into the Cayley plane. Appendix \ref{ap:roots} collects some facts about the structure of the Lie algebra $\f_4$ of $F_4$ and its fundamental representation and Appendix \ref{ap:nilpotent} gives a very brief introduction to the classification of nilpotent orbits in Lie algebras and show how this applies to $\f_4$.

A few words about conventions: throughout this paper,  we denote the complexification $V\otimes_\rn\cn$ of a  real vector space or vector bundle $V$ by $V^\cn$. Unless otherwise mentioned, the symbol $G$ will always denote a compact semi-simple Lie group with trivial centre contained in $\U(n)$ for some $n$, while $\g$ will denote its Lie algebra.

\section{The Cayley plane}\label{se:Cayley}

In this section we recall some classical facts on the construction of the Cayley plane $\Oc P^2$ from an exceptional Jordan algebra. We also exhibit the Cayley plane as a subset of a Grassmannian, a construction we will find useful when studying the twistor theory of $\Oc P^2$. Our main references are \cite{adams, baez, Ha}.

We denote by $\Oc$ the $8$-dimensional division algebra of octonions. For $x\in\Oc$, we denote by $x^*$ the octonionic conjugation of $x$. The following identities hold for all $u, v\in\Oc$ and express the fact that $\Oc$ is an \emph{alternative, nicely normed $*$-algebra} \cite{adams,baez}:
\begin{equation}\label{alternative}
\begin{array}{ccc}
(uu)v=u(uv), &(uv)u=u(vu),& (vu)u=v(uu), \\
(uu^*)v=u(u^*v), &(uv)u^*=u(vu^*),& (vu^*)u=v(u^*u).
\end{array}
\end{equation}

Denote by $\h_3(\Oc)$ the $27$-dimensional simple formally real Jordan algebra of hermitian $3\times 3$ matrices with entries in $\Oc$, equipped with the product
$$
a\circ b=\frac12(ab+ba)\qquad(a, b\in\h_3(\Oc)).
$$
We also equip $\h_3(\Oc)$ with the inner product
$$
\langle a,b\rangle_\rn=\frac12\tr(a\circ b)\qquad(a, b\in\h_3(\Oc)).
$$
As is well known, the exceptional Lie group $F_4$ is realised as the automorphism group of  $\h_3(\Oc)$ with respect to the product $\circ$. Since $F_4$ also preserves the trace, we have $F_4\subset  \SO(\h_3(\Oc))$. Moreover, since the elements in $F_4$ fix the real subspace generated by the identity matrix in $\h_3(\Oc)$, they  preserve also the $26$-dimensional subspace $\h^0_3(\Oc)$ of the traceless matrices in $\h_3(\Oc)$. This space is the \emph{fundamental representation} of $F_4$, its smallest non-trivial representation, see Appendix \ref{ap:roots}. For $a, b\in\h_3(\Oc)$, we denote by $a\cdot b$ the traceless part of the product $a\circ b$.

The \emph{Cayley plane} $\Oc P^2$ is the $16$-dimensional manifold of matrices $P\in\h_3(\Oc)$ satisfying $P^2=P$ and $\tr(P)=1$. The inner product $\langle \cdot,\cdot\rangle_\rn$ on $\h_3(\Oc)$ induces a  Riemannian metric on $\Oc P^2$. The isometry group of $\Oc P^2$ is precisely $F_4$, acting transitively with isotropy subgroups conjugated to $\Spin(9)$.  Hence $\Oc P^2\cong F_4/\Spin(9)$, a compact Riemannian symmetric space of rank 1.

For $P\in\Oc P^2$, let $L_P$ denote the linear endomorphism of $\h_3(\Oc)$ given by
$$
L_P(X)=P\circ X\qquad(X\in\h_3(\Oc)).
$$
This gives a eigenspace decomposition
\begin{equation}\label{dec}
\h_3(\Oc)=A_0(P)\oplus A_{\frac12}(P)\oplus  A_{1}(P),
\end{equation}
where $A_j(P)$ is the eigenspace associated to the eigenvalue $j$, with $j=0,\frac12,1$. For example, if $P\in \Oc P^2$  is the diagonal matrix $\mathrm{diag}(1,0,0)$, then $A_1(P)$ is just the real span of $P$ while
\begin{align*}\label{A}
A_{0}(P)&=\Bigg\{\left(
       \begin{array}{ccc}
       0 & 0 & 0 \\
        0 & \beta & z \\
        0 & z^* & \gamma \\
      \end{array}
    \right): \beta,\gamma\in\rn,\,z\in\Oc\Bigg\},\\
     A_{\frac12}(P)&=\Bigg\{\left(
     \begin{array}{ccc}
        0 & x & y \\
        x^* & 0 & 0 \\
        y^* & 0 & 0 \\
      \end{array}
    \right): x,y\in\Oc\Bigg\}.
\end{align*}

Y. Huang and N. C. Leung \cite{HL} gave a uniform description of all compact Riemannian symmetric spaces as  Grassmannians. According to their description, $\Oc P^2$ is the space of all copies of $\h_2(\Oc)$ in $\h_3(\Oc)$. As we shall see, $\Oc P^2$  can also be interpreted as a certain Grassmannian of subspaces of $\h^0_3(\Oc)$. Our proof is independent of that in \cite{HL}.

\begin{theorem}\label{grass} There is an isometry between $\Oc P^2$ and the Grassmannian $Gr^a$ of $10$-dimensional subspaces $V$ of $\h^0_3(\Oc)$ satisfying $V\cdot V=V$.
\end{theorem}

\begin{proof} Take $V\in Gr^a$ and let $\hat{V}$ be the maximal subalgebra of $\h_3(\Oc)$ containing the $11$-dimensional subalgebra $\rn I\oplus V$. According to Racine's classification of maximal subalgebras of exceptional Jordan algebras \cite{Ra}, either $\hat{V}$ is isomorphic to the $15$-dimensional Jordan algebra  $\h_3(\hn)$ or $\hat{V}$ is an $11$-dimensional subalgebra of the form
$$
\hat V_P= A_0(P)\oplus A_1(P),
$$
for some $P\in \Oc P^2$. However, as the maximal subalgebras of $\h_3(\hn)$ have dimension $7$ and $9$ (see \cite{Ra}), we must have $\hat{V}=\hat V_P$ for some $P\in \Oc P^2$ .

Assume now that $\hat V_P=\hat V_{P'}$ for $P, P'\in \Oc P^2$. In view of the orthogonal eigenspace decomposition \eqref{dec}, we must have $A_{\frac12}(P)=A_{\frac12}(P')$. On the other hand, by transitivity of the $F_4$ action on $\Oc P^2$, we can without loss of generality assume that $P=P_0=\mathrm{diag}(1,0,0)$. But a straightforward computation now shows that $A_{\frac12}(P')=A_{\frac12}(P)$ for $P'\in\Oc P^2$ is only possible if $P'=P_0$.

Thus, given $V\in Gr^a$, we have seen that there exists a unique $P\in \Oc P^2$ such that $\hat{V}=\hat V_P$. This defines an injection $\chi: Gr^a\to\Oc P^2$ mapping this $V$ to $P$.

Conversely, given $P\in\Oc P^2$, consider the $11$-dimensional maximal subalgebra $\hat V_P= A_0(P)\oplus A_1(P)$. Take $g\in F_4$ such that $P=g(P_0)$, where $P_0=\mathrm{diag}(1,0,0)$. Clearly we have $\hat V_P=g(\hat V_{P_0})$ and $\rn I\subset  \hat V_{P_0}$. Since $F_4$ preserves $\rn I$ we also have $\rn{I}\subset  \hat V_{P}$. Set
$$
V=\hat V_P\cap(\rn I)^\perp.
$$
It is now easy to see that $V\in Gr^a$ and $\chi(V)=P$, thus $\chi$ is also surjective. Since $\chi$ is $F_4$-invariant, $\chi$ is an isometry.
\end{proof}

Interpreting $\Oc P^2$ as the submanifold of $\h_3(\Oc)$ consisting of the projection matrices $P$ with $\tr(P)=1$, the tangent plane at $P$ to $\Oc P^2$ is given by
\begin{equation*}\label{TP}
T_P\Oc P^2=\{Z\in\h_3(\Oc):\, 2P\circ Z=Z  \}=A_{\frac12}(P).
\end{equation*}
On the other hand, given a projection matrix $P\in \Oc P^2$, let us denote by $V_P$ the corresponding subspace in $Gr^a$. With this identification, a tangent vector  $Z\in T_P\Oc P^2$ corresponds to an element
$$
\hat{Z}\in\mathrm{Hom}(V_P,V_P^\perp)\oplus \mathrm{Hom}(V_P^\perp,V_P).
$$
We can make this correspondence more explicit as follows.
\begin{proposition} For any $Z\in T_P\Oc P^2$ and $X\in\h^0_3(\Oc)$, we have
\begin{equation}\label{hatZ}
\hat{Z}(X)=-4 Z\circ X+4Z\circ (P\circ X)+4P\circ(Z\circ X).
\end{equation}
\end{proposition}

\begin{proof} Given a projection matrix $P\in \Oc P^2$ it follows easily that the orthogonal projection onto $V_P$ is given by
\begin{equation}\label{ortoproj}
X\mapsto X-4P\circ X+4P\circ (P\circ X)\qquad(X\in\h_3^0(\Oc)).
\end{equation}
Differentiating \eqref{ortoproj}  at $P$ in the direction $Z\in T_P\Oc P^2$, we obtain \eqref{hatZ}.
\end{proof}

\begin{proposition}\label{isotropic} Suppose that the tangent vector $Z\in T_P^\cn\Oc P^2$ is isotropic, that is, $\tr(Z\circ Z)=0$. Then $\hat{Z}$ defined by \eqref{hatZ} is nilpotent with nilorder  $r\leq 5$.
\end{proposition}

\begin{proof} Without loss of generality we may again assume that $P=\mathrm{diag}(1,0,0)$. Since $\hat{Z}$ transforms vectors of $V_P$ in vectors of $V_P^\perp$ and vice-versa, it is enough to prove that $\hat{Z}^4=0$ when restricted to $V_P^\perp$. Up to multiplication by a constant, we have
\begin{equation*}
\begin{split}
\label{Z2} \hat{Z}^2(X)=&Z\circ(Z\circ X)-Z\circ(P\circ (Z\circ X))-Z\circ (Z\circ(P\circ X))\\
&-P\circ(Z\circ(Z\circ X))+P\circ (Z\circ (P\circ (Z\circ X)))\\
&+ P\circ (Z\circ (Z\circ (P\circ X)))
-Z\circ (P\circ (Z\circ X)))\\
& +Z\circ (P\circ (P\circ (Z\circ X)))+Z\circ (P\circ(Z\circ(P\circ X))).
\end{split}
\end{equation*}
This can be considerably simplified if we take $X\in V_P^\perp=A_{\frac12}(P)$. In view of
$$
A_{\frac12}(P)\circ A_{\frac12}(P)\subseteq A_{0}(P)\oplus A_{1}(P),\,\, A_{\frac12}(P)\circ ( A_{0}(P)\oplus A_{1}(P))\subseteq  A_{\frac12}(P),
$$
and recalling that $Z\in T_P^\cn\Oc P^2=A_{\frac12}(P)^\cn$, we see that, up to multiplication by a constant, $\hat{Z}^2(X)=Z\circ(Z\circ X)$ for $X\in V_P^\perp$. Hence
$$
\hat{Z}^4(X)=Z\circ(Z\circ (Z\circ (Z\circ X).
$$

So, it remains to prove that $L_Z^4(X)=0$, where $L_Z$ is the linear endomorphism of $\h_3(\Oc)$ given by the left multiplication by $Z$. Let us write
$$
Z=\left(
\begin{array}{ccc}
0 & x & y \\
x^* & 0 & 0 \\
y^* & 0 & 0 \\
\end{array}
\right)
$$
for some $x,y\in\Oc\otimes_\rn\cn$. It is easy to see that $Z$ is isotropic if and only if $xx^*+yy^*=0$. Since $\Spin(9)\subset F_4$ acts transitively on the space of real lines in $T_P\Oc P^2$, we can assume without loss of generality that
$$
X=\left(
\begin{array}{ccc}
0 & 1 & 0 \\
1 & 0 & 0 \\
0 & 0 & 0 \\
\end{array}
\right)
$$
Taking also \eqref{alternative} into account, a straightforward calculation now shows that $L_Z^4(X)=0$.
\end{proof}

\section{The twistor spaces of  $\Oc P^2$}\label{se:twistor}

Twistors first appeared in the context of theoretical physics, in particular in the work of Penrose, see for example \cite{PenroseRindler}.  For harmonic maps, twistor constructions can be found already in the works of Chern \cite{Chern} and Calabi \cite{Calabi}, and have since been used by many others as the field has developed. The main idea behind twistor constructions of harmonic maps into a Riemannian manifold $(N,h)$ is to find an (almost) complex manifold $(Z, J)$ and a fibration $\pi:Z\to N$ with the following property: for any Riemann surface $M$ and any (almost) holomorphic map
$$
\psi:M\to Z,
$$
the composition $\pi\circ\psi:M\to N$ is harmonic. If that is the case, then $\pi:(Z,J)\to(N,h)$ is called a \emph{twistor fibration} and $(Z,J)$ a \emph{twistor space} for $(N,h)$. Using twistor fibrations we can thus construct harmonic maps from (almost) holomorphic maps.

F.E. Burstall and J.H. Rawnsley introduced  a general twistor theory for harmonic maps into compact symmetric spaces \cite{BurstallRawnsley}. In this section we will briefly recall their construction, and then study this in the specific context of the compact symmetric space $\Oc P^2$. For notation regarding the weight spaces of the fundamental representation of $F_4$ we refer to Appendix \ref{ap:roots}.

\subsection{Harmonic maps into Lie groups and symmetric spaces}

A map $\varphi:(M,g)\to(N,h)$ between two Riemannian manifolds is said to be \emph{harmonic} if it satisfies the harmonic map equation
$$
\tr\nabla\d\varphi = 0,
$$
where $\nabla$ is the connection on $\varphi^{-1}TN\otimes T^*M$ induced by the Levi-Civita connections on $(M,g)$ and $(N,h)$ respectively. As is easily seen when $\dim M=2$, this equation depends only on the conformal structure on $M$. Thus, the concept of a harmonic map from a Riemann surface to a Riemannian manifold makes sense.

Let $G$ be a compact semi-simple Lie group equipped with a bi-invariant metric and $M$ a Riemann surface. For simplicity of exposition, we will assume from now on that $G$ is semi-simple, has trivial centre and is contained in $\U(n)$ for some $n$. For a map $\varphi:M\to G$, set
$$
A^\varphi=\frac12 \varphi^{-1}\d\varphi.
$$
For any local complex coordinate $z$ on $M$, we write
$$
A^\varphi=A_z^\varphi\d z + A_{\bar z}^\varphi\d\bar z.
$$
The \emph{integrability equation}
\begin{equation}\label{eq:integrability}
(A^\varphi_z)_{\bar{z}}-(A^\varphi_{\bar{z}})_z= 2[A^\varphi_z, A^\varphi_{\bar{z}}]
\end{equation}
follows easily from the Maurer-Cartan equation. Furthermore, it is easy to see that
$\varphi$ is harmonic if and only if
\begin{equation}\label{eq:harmonicity}
(A_z^\varphi)_{\bar z} + (A_{\bar z}^\varphi)_z = 0.
\end{equation}

Define the connection $D^\varphi=\d+A^\varphi$ on the trivial bundle $M\times\cn^n$. The condition for $\varphi$ being harmonic can now be written as
$$
D_{\bar z}^\varphi A_z^\varphi=  A_z^\varphi D_{\bar z}^\varphi,
$$
where $D_{\bar z}^\varphi=\partial_{\bar z} + A_{\bar z}^\varphi$. Thus the harmonicity of $\varphi$ is equivalent to $A^\varphi_z$ being a holomorphic endomorphism on the trivial bundle with respect to the holomorphic structure induced by $D^\varphi$.

As is well known, a compact inner symmetric space $G/K$ may be totally geodesically immersed in $G$ as a connected component of $\sqrt{e}=\{g\in G:\, g^2=e\}$.  When a harmonic map $\varphi:M\to G$ takes values in a connected component of  $\sqrt{e}$, we can write $\varphi=\pi_{\varphi}-\pi_{\varphi}^\perp$, where  $\pi_{\varphi}$ denotes the orthogonal projection onto the vector bundle, also denoted by $\varphi$, whose fiber at $z$ is the $(+1)$-eigenspace of $\varphi(z)$.

\subsection{General twistor theory.}\label{sec:twistor}

Let $\g$ be the Lie algebra of $G$ and $\t\subset\g$ be a maximal torus. Denote by $\Delta\subset\i\t^*$ the set of roots, where $\i=\sqrt{-1}$, and for $\alpha\in\Delta$ let $\g^\alpha$ be the corresponding root space. Fix a set of positive simple roots $\Phi^+=\{\alpha_1,\ldots,\alpha_n\}$ with dual basis $H_1,\ldots, H_n\in\t$ so that $\alpha_i(H_j)=\i\,\delta_{ij}$.

For any non-empty subset $I\subset\{1,\ldots,n\}$ consider the \emph{canonical element} $\xi_I=\sum_{i\in I}H_i$. From this we construct a parabolic subalgebra of $\g^\cn$ as follows: denote by $\g^{\xi_I}_j$ the $j$-eigenspace of $\ad(\xi_I/\i)$ and set
$$
\p_I=\sum_{j\geq0}\g^{\xi_I}_j.
$$
Let $P_I$ be the corresponding parabolic subgroup and $T_I=G^{\cn}/P_I=G/P_I\cap G$ the corresponding flag manifold.

The subset $I$ also defines the inner involution
$$
\tau_I=\Ad \exp(\pi\xi_I),
$$
which corresponds to the symmetric space $N_I=G/K$, where $P_I\cap G\subseteq K$. The corresponding  symmetric decomposition is
\begin{equation}\label{symdecomposition}
\g=\k\oplus \m,\qquad \k^\cn=\sum_{i\,\,\mathrm{even}}\g^{\xi_I}_i,\qquad \m^\cn=\sum_{i\,\,\mathrm{odd}} \g^{\xi_I}_i.
\end{equation}
Since $P_I\cap G\subset K$ we have a homogenous projection $p_I:T_I\to N_I$, which we refer to as a \emph{canonical fibration} \cite{BurstallRawnsley}. The horizontal and vertical spaces at the base point $x_0\in T_I$ (that we identify with the identity coset) of the canonical fibration are given by
\begin{equation}\label{verticalhorizontal}
\V^\cn(T_I)=\k^\cn\cap (\g_0^{\xi_I})^\perp,\qquad \H^\cn(T_I)=\m^\cn.
\end{equation}
Moreover, the flag manifold $T_I$ carries a complex structure $J_1$ with $(1,0)$-space at the base point $x_0$ given by $T_{J_1}^{1,0}T_I=\sum_{i>0}\g^{\xi_I}_i.$
By reversing the orientation of $J_1$  on the fibres, we obtain  another almost complex structure on $T_I$, usually denoted by $J_2$:
\begin{equation}\label{J210}
T_{J_2}^{1,0}T_I=\sum_{i\,\,\mathrm{even}<0}\g^{\xi_I}_i\oplus  \sum_{i\,\,\mathrm{odd}>0}\g^{\xi_I}_i.
\end{equation}
Thus the first summand on the right-hand side is $\V^{1,0}_{J_2}$ and the second summand is $\H^{1,0}_{J_2}$.

\begin{proposition}\label{prop:twistor}\cite{BurstallRawnsley}
The canonical homogenous map $p_I:T_I\to N_I$ is a twistor fibration with respect to the almost complex structure $J_2$.
\end{proposition}

For later use we also note that that $P_I\cap G$ is precisely the fixed set for the automorphism
\begin{equation*}\label{sigmaI}
\sigma_I=\Ad \exp(\frac{2\pi}{k}\xi_I)
\end{equation*}
of order $k$, where
$$
k=\max\{\alpha(\xi_I)/\i :\, \alpha\in\Phi^+\} + 1.
$$
Hence the flag manifold $T_I$ also carries a structure of a $k$-symmetric space. Let $\omega$ be the primitive $k$-th root of unity and for each $j=0,\dots,k-1$ let
\begin{equation}\label{gj}
\g^j_{\xi_I}=\sum_{i=j~\mathrm{mod}~k}\g_i^{\xi_I}
\end{equation}
be the eigenspace of $\sigma_I$ with eigenvalue $\omega^j$.

\subsection{The canonical fibrations of $\Oc P^2$.}\label{Sec:canonicalprojections}

We are now in a position to study the canonical projections for the Cayley plane. We follow the notation set out in Appendix \ref{ap:roots}. It is easy to see that there are exactly three flag manifolds fibering canonically over $\Oc P^2$ \cite[\S 7]{CP2015}, corresponding to the following subsets of simple roots:
\begin{equation}\label{eq:I}
I=\{\alpha_3\},\quad I=\{\alpha_4\},\quad I=\{\alpha_3, \alpha_4\}.
\end{equation}
We denote the corresponding flag manifolds $T_3$, $T_4$ and $T_{34}$, respectively. It is easy to see that
\begin{equation*}
\begin{split}
T_3=&F_4/(\SU(3)\times\U(1)\times\SU(2))\\
T_4=&F_4/(\Spin(7)\times\U(1))\\
T_{34}=&F_4/(\SU(3)\times\U(1)\times\U(1)).
\end{split}
\end{equation*}
%
These spaces may be interpreted geometrically as follows.

\begin{proposition}\label{P4}(\cite{LM}, Proposition 6.6) The flag manifold $T_4$ is the space of all lines in $\h_3^0(\Oc)^\cn$ that square to zero, i.e.,
$$
T_4=\{\ell \in \P(\h_3^0(\Oc)^\cn):\, \ell^2=0\}.
$$
\end{proposition}

It is easy so see that any $\ell\in T_4$ is isotropic, so that $T_4$ is contained in the quadric of isotropic lines in $\h_3^0(\Oc)^\cn$. 

\begin{proposition}\label{P3}(\cite{LM}, Proposition 6.7)
The flag manifold $T_3$ is the space of two-dimensional isotropic subspaces $\C\subset \h_3^0(\Oc)^\cn$ satisfying $\C^2=0$  (with respect to the usual matrix product).
\end{proposition}

\begin{proposition} The flag manifold $T_{34}$ is the space of all pairs $(\ell,\C)$, where $\ell\in T_4$, $\C\in T_3$ and $\ell\subset\C$.
\end{proposition}


\begin{proof} Given such a pair $(\ell,\C)$, the stabilizer of this pair in $F_4$ is the stabilizer of $\ell$ in the stabilizer of $\C$. The latter is $\SU(3)\times\U(1)\times\SU(2)$ where the first factor acts trivially on $\C$ and the third factor acts by the standard $2$-dimensional representation on $\C$. Since this is transitive on the lines in $\C$ with stabilizer $\U(1)$, we conclude that $F_4$ is transitive on such pairs with stabilizer $\SU(3)\times\U(1)\times\U(1)$.
\end{proof}

We may also embed $\Oc P^2$ as the connected component of $\sqrt{e}=\{g\in F_4 :\, g^2=e\}$ containing $\exp(\pi\xi_I)$, where $I$ is one of the sets in \eqref{eq:I}. This is a totally geodesic embedding of $\Oc P^2$ into $F_4$. Similarly, $T_I$ can be identified with the connected component of $\sqrt[k]{e}=\{g\in F_4 :\ g^k=e\}$ containing the element
$$
\exp\big(\frac{2\pi}{k}\xi_I\big).
$$
The Lie group $F_4$ acts on both the images of $\Oc P^2$ and $T_I$ by conjugation, and
$$
p_I\big(\exp \big(\frac{2\pi}{k} \xi_I\big)\big)=\exp (\pi \xi_I).
$$
This can be made more explicit via the fundamental representation, henceforth denoted by $\rho$, of $F_4$ and $\f_4$ (see Appendix \ref{ap:roots}). Recall the isometry $Gr^a\cong\Oc P^2$ from Theorem \ref{grass}. The subspace in $Gr^a$ corresponding to $\exp(\pi\xi_I)$ is exactly the $(+1)$-eigenspace of $\rho(\exp(\pi\xi_I)$. Denoting this space by $V$ we have
$$
\rho(\exp(\pi\xi_I))=\pi_V-\pi_{V^\perp},
$$
where $\pi_W$ denotes orthogonal projection onto a subspace $W$.

Let us first consider  the $3$-symmetric space $T_4$, the flag manifold of isotropic lines $\ell$ in $\h^0_3(\Oc)^\cn$ satisfying $\ell^2=0$. The parabolic group $P_4$ stabilizes  the isotropic line $\ell= W_{12}$ which we thus choose as our base point. We now have
$$
\rho\left(H_4\right)=\begin{cases} 0, & \text{on}\, W_0\oplus W_1\oplus W_2\oplus W_3\\ ~\mathrm{i}, & \text{on}\, W_4\oplus W_5\oplus W_6\oplus W_7\oplus W_8\oplus W_9\oplus W_{10}\oplus W_{11}\\ 2\mathrm{i}, & \text{on}\,W_{12} \end{cases}.
$$
It follows from the multiplication table in Appendix \ref{ap:roots} that
$$
\ell_s=\sum_{i\geq -3} W_i
$$
is the \emph{stabilizer} of $\ell$, i.e., the maximal subspace satisfying $\ell_s\cdot \ell\subseteq \ell$. The $(+1)$-eigenspace of $\rho(\exp(\pi H_4))$ is  $\ell\oplus\bar\ell\oplus (\ell_s\cap \overline{\ell}_s)$ and we thus conclude that the projection $p_4:T_4\to \Oc P^2$ is given by
\begin{equation}\label{p4}
p_4(\ell)=\ell\oplus \overline \ell \oplus (\ell_s\cap \overline{\ell}_s)\qquad(\ell\in T_4).
\end{equation}

We now consider the $4$-symmetric space $T_3$. The parabolic group $P_3$ stabilizes  the isotropic $2$-dimensional subspace $\C=W_{11}\oplus W_{12}$ which thus corresponds to our base point. We have
$$
\rho\left(H_3\right)=\begin{cases} 0,&\text{on}\, W_0\oplus W_4\\ ~\mathrm{i},&\text{on}\, W_1\oplus W_2\oplus W_3\oplus W_5\oplus W_6\oplus W_7\\ 2\mathrm{i},& \text{on}\, W_8\oplus W_9\oplus W_{10}\\ 3\mathrm{i},&\text{on}\, W_{11}\oplus W_{12} \end{cases}.
$$
Observe that the stabilizer $\C_s$ and the \emph{annihilator} of $\C$, i.e., the maximal subspace $\C_a$ satisfying $\C_a\cdot\C=0$, are given by
$$
\C_s=W_{-4}\oplus\sum_{i\geq 0} W_i\quad \text{and}\quad \C_a=\sum_{i >0, i\neq 4} W_i,
$$
respectively. We also have $\C_a^2=\sum_{i\geq 8}W_i$ and we thus conclude that the projection $p_3:T_3\to \Oc P^2$ is given by
\begin{equation}\label{p3projection}
p_3(\C)=(\C_a^2\cap \C^\perp)\oplus(\C_s\cap \overline{\C}_s)\oplus \overline{(\C_a^2\cap \C^\perp)}\qquad(\C\in T_3).
\end{equation}

The flag manifold $T_{34}$ is a $6$-symmetric space and the parabolic group $P_3\cap P_4=P_{34}$ stablizes the pair $(\ell,\C)=(W_{12},W_{11}\oplus W_{12})$ which we thus take as our base point. We have
 $$
 \rho\left(H_3+H_4\right)=\begin{cases} 0,&\text{on}\, W_0\\ ~\mathrm{i},&\text{on}\, W_1\oplus W_2\oplus W_3\oplus W_4\\ 2\mathrm{i},&\text{on}\, W_5\oplus W_6\oplus W_7\\ 3\mathrm{i},&\text{on}\, W_8\oplus W_9\oplus W_{10}\\ 4\mathrm{i},&\text{on}\, W_{11}\\ 5\mathrm{i},&\text{on}\, W_{12} \end{cases}.
 $$
 Let $\C_\ell$ be the maximal subspaces satisfying $\C_\ell\cdot \C\subseteq \ell$, so that
 $$
 \C_\ell=\sum_{i\geq 1}W_i,\ \C_\ell^2=\sum_{i\geq 5}W_i\ \text{and}\ \C_\ell^3=\sum_{i\geq 8}W_i.
 $$
By identifying the $(+1)$-eigenspace of $\rho(\exp(\pi (H_3+H_4))$ we thus conclude that the projection $p_{34}:T_3\to \Oc P^2$ is given by
\begin{multline}\label{p34projection}
p_{34}(\ell,\C)=(\C\cap \ell^\perp)\oplus (\C_\ell^2\cap{\C_\ell^3}^\perp)\oplus (\overline{\C_\ell}^\perp\cap \C_\ell^\perp)\\
\oplus \overline{(\C\cap \ell^\perp)\oplus (\C_\ell^2\cap{\C_\ell^3}^\perp)}\qquad((\ell,\C)\in T_{34}).
\end{multline}

 \subsection{Twistor lifts}\label{subse:harmonic}

In this section we investigate the condition of a map into one of the $T_I$ to be holomorphic with respect to the almost complex structure $J_2$ introduced in the previous section. Using Proposition \ref{prop:twistor} this will then imply harmonicity of the resulting map into $\Oc P^2$. To avoid cumbersome notation,  we will denote the trivial bundle $M\times\h_3^0(\Oc)^\cn$ by just $\h_3^0(\Oc)^\cn$.

\begin{lemma}
Let $\psi:M \to T_I=F_4^\cn/P_I$ be a smooth map and set $\varphi=p_I\circ \psi:M\to\Oc P^2$. Then $\psi$ is $J_2$-holomorphic (so that $\varphi$ is harmonic) if and only if:
\begin{enumerate}
\item[(i)] when $I=\{3\}$,  $\psi=\C$ is a rank 2 holomorphic subbundle of $\h^0_3(\Oc)^\cn$ with respect to  $D^\varphi_{\bar z}$, contained in $\ker A^{\varphi}_z$.
\item[(ii)] when $I=\{4\}$, $\psi=\ell$ is a holomorphic line subbundle of $\h^0_3(\Oc)^\cn$ with respect to  $D^\varphi_{\bar z}$, contained in $\ker A^{\varphi}_z$.
\item[(iii)] when $I=\{3,4\}$, $\psi=(\ell,\C)$, $\ell$ and $\C$ are holomorphic subbundles of $\h^0_3(\Oc)^\cn$ with respect to $D^\varphi_{\bar z}$, with $\ell$ contained in $\ker A^{\varphi}_z$ and $A_z^\varphi(\C)\subset \ell$.
\end{enumerate}
\end{lemma}

\begin{proof}We will prove the lemma for  $I=\{4\}$, the remaining cases are similar.

Assume first that $\psi=\ell$ is $J_2$-holomorphic. Let us fix a  point $z_0\in M$ and, without loss of generality, assume that $\psi(z_0)=W_{12}$. In view of \eqref{verticalhorizontal} and \eqref{J210}, it is clear that $\rho(\H_{J_2}^{1,0})W_{12}=0$. We also have $\rho(\V_{J_2}^{1,0})(\varphi(z_0)\cap W_{12}^\perp)\perp W_{12}$, because $\V_{J_2}^{1,0}$ is a direct sum of negative root spaces.

The condition $\rho(\H_{J_2}^{1,0})W_{12}=0$ means that $\psi=\ell$ lies in $\ker A^{\varphi}_z$ at $z_0$, while the condition $\rho(\V_{J_2}^{1,0})(\varphi(z_0)\cap W_{12}^\perp)\perp W_{12}$  means that $D^\varphi_{\bar z}(\ell)\subset\ell$. Consequently $\ell$ is a holomorphic line bundle with respect to $D_{\bar{z}}^\varphi$ contained in $\ker A^{\varphi}_z$.

For the converse, we can use the Lie theoretic description  of the fundamental representation of $F_4$  (see Appendix \ref{ap:roots}) to check the following: (a) for each nonzero $ X\in  \H_{J_2}^{0,1}$, $\rho(X)W_{12}\neq 0 $; (b) for each nonzero $X\in  \V_{J_2}^{0,1}$,  $\rho(X)(\varphi(z_0)\cap W_{12}^\perp)=W_{12}$. From (a)  we see that if $\ell$ lies in $\ker A^{\varphi}_z$, then
the component of $\ell_z$ along $\H_{J_2}^{0,1}$ must vanish everywhere. From (b) we see that if $\ell$ is a holomorphic line subbundle of $\h^0_3(\Oc)^\cn$  with respect to  $D^\varphi_{\bar z}$, then  the component of $\ell_z$ along $\V_{J_2}^{0,1}$ must vanish everywhere.
\end{proof}

\section{Nilconformal Harmonic Maps}\label{se:nilconformal}

Recall that a harmonic map $\varphi:M\to G\subset \U(n)$ is said to be \emph{nilconformal} if $A^\varphi_z$ is a nilpotent element of the Lie algebra $\g^\cn$ at each point. It follows from the holomorphicity of $A^\varphi_z$ that any harmonic map $\varphi:S^2\to G$ is nilconformal. By Engel's theorem we know that $A^\varphi_z$  is also nilpotent at each point as a complex linear endomorphism of  $\cn^n$.

The complex algebraic variety of nilpotent elements of $\g^\cn$, which we denote by $\N(\g^\cn)$,  decomposes into a disjoint union of conjugacy classes $\O_X$ of nilpotent elements  under the adjoint action of $G^\cn$ (the \emph{nilpotent $G^\cn$-orbits}). A basic fact from the theory of algebraic groups is that each orbit $\O_X$ is open in its Zariski closure $\overline{\O_X}$, and the latter is a complex algebraic subvariety. Since $A^\varphi_z$ is holomorphic and $M$ has complex dimension $1$, we have the following result.
\begin{theorem} Let $\varphi:M\to G$ be  a nilconformal harmonic map. Then, off a discrete subset of $M$, $A^\varphi_z$ takes values in a single nilpotent $G^\cn$-orbit $\O_X$.
\end{theorem}
The least $r$ such that $X^r=0$ is the \emph{nilorder} of $X\in \N(\g^\cn)$. If  $\varphi:M\to \O_X$ off a discrete set and $r$ is the nilorder of $X$, we say that $\varphi$ has \emph{nilorder} $r$.

Consider now a symmetric space $G/K$ totally geodesically immersed in  $G$ and let $\g=\k\oplus \m$ be the corresponding symmetric decomposition. Suppose that $\varphi$ corresponds to a nilconformal harmonic map with values in $G/K$. In this case, $A^\varphi_z$ is a holomorphic section of the bundle  $[\m]^\cn$ whose fibre at the point $gK$ is $\Ad g(\m^\cn)$ . Let $\psi:M\to G$ be a (local) lift of $\varphi$ into $G$ so that $\varphi=\psi K$. Then $\Ad \psi^{-1}(A^\varphi_z )$ takes values in a single conjugacy class of nilpotent elements in $\m^\cn$ under the adjoint action of $K^\cn$.

\begin{proposition}\label{prop:nilconformal} Any weakly conformal harmonic map from a Riemann surface to $\Oc P^2$ is nilconformal with nilorder $r\leq 5$.
\end{proposition}

\begin{proof} A weakly conformal harmonic map from a Riemann surface $M$ is a conformal immersion off a discrete set of points where its differential vanishes. Hence the result follows from Proposition \ref{isotropic}.
\end{proof}

Making use  of the classification of nilpotent orbits in Lie algebras as described in \cite{Collingwood}, we see that there are precisely two such nontrivial nilpotent orbits  associated to the Cayley plane (see Appendix \ref{ap:nilpotent} for more details): if we consider the base point $V_0\in Gr^a\cong \Oc P^2$ given, in terms of weight spaces, by
\begin{equation}\label{basepoint}
V_0=W_{11}+ W_{7}+ W_{6}+ W_{5}+ W_0+ W_{-5}+ W_{-6}+ W_{-7}+ W_{-11},
\end{equation}
and the corresponding symmetric decomposition  $\f_4=\k\oplus \m$  given by \eqref{symdecomposition}, with $I=\{3,4\}$, then one of the orbits is represented by  $X_4\in \m^\cn$ and the other orbit  represented by $X_3+X_4\in \m^\cn$, where $X_3$ and $X_4$ are nonzero elements of the root spaces $(\f_4)_{\alpha_3}$ and $(\f_4)_{\alpha_4}$, respectively.

\subsection{$J_2$-lifts}\label{J2lifts}

Suppose that we have a nilconformal harmonic map $\varphi:M\to \Oc P^2$. Then, off a discrete subset of $M$, $A^\varphi_z$ takes values in a single $F^\cn_4$-orbit of nilpotent elements in $\f^\cn_4$. More precisely,  $A^\varphi_z$  is a section of $[\m^\cn]\cap \overline{\O_X}$ with   $X=X_4$ or $X=X_3+X_4$.
We will show that in both cases (see \S \ref{XX4} and \S \ref{XX34}) $\varphi$ admits a $J_2$-holomorphic lift into $T_4$, so that the following holds.
\begin{theorem}\label{thm-twistor}
Any weakly conformal harmonic map $\varphi:M\to \Oc P^2$ admits a $J_2$-holomorphic lift into $T_4$.
\end{theorem}

\subsubsection{The case $X=X_4$.}\label{XX4} In this case, $X$  acts as follows with respect to the fundamental representation:

$$
\begin{array}{ccccc}
W_{-12}\mapsto W_{-11} & W_{-7}\mapsto W_{-3}  & W_{-6}\mapsto W_{-2} & W_{-5}\mapsto W_{-1}  & W_{-4}\mapsto W_0\\
W_0\mapsto W_4 & W_1\mapsto W_5  & W_2\mapsto W_6 & W_3\mapsto W_7  & W_{11}\mapsto W_{12}\\
\end{array}
$$
and $X(W_i)=0$ otherwise. We note from this that
$$
\im X^2 \subseteq W_4\subset V_0^\perp,
$$
with $V_0$ given by \eqref{basepoint}. 
\begin{lemma}
$X$ has nilorder $3$,  $\di \im X^2=1$.
\end{lemma}

\begin{proof}
Clearly, $X^3=0$ and $X^2(W_i)=0$  for all $i\neq -4$.
However, since $\alpha_4$ is a weight, we cannot have $X^2=0$, for this would imply that $W_{-4}+X(W_{-4})$ is a representation of  the $\mathfrak{sl}_2$ associated to the root $\alpha_4$, which is impossible since the eigenvalues must be symmetric around the origin.
\end{proof}

Let $(W_4)_a$ be the \emph{annihilator} of $W_4 = \im X^2$, i.e., the maximal subspace $W$ of $\h_3^0(\Oc)^\cn$ satisfying $W\cdot W_4=0$. It follows easily that
\begin{equation*}
\begin{split}
(W_4)_a=&W_{-11}+W_{-10}+W_{-9}+W_{-8}+W_{-3}+W_{-2}+W_{-1}+W'_0\\
&+W_4+W_5+W_6+W_7+W_8+W_9+W_{10}+W_{12},
\end{split}
\end{equation*}
where $W'_0=(W_4)_a\cap W_0$.  From a simple calculation using the symmetry of the trilinear form on $\h_0^3(\Oc)^\cn$ it follows easily that $W'_0=\overline{W'_0}$. Set $V=(W_4)_a\cap(W_4)_a^2\cap\varphi$ and observe that
$$
V\cap \bar{V}^\perp=W_{-11}+W_7+W_6+W_5.
$$
From this we also see that $(V\cap\bar{V}^\perp)^2=0$.

Since $A^\varphi_z$  is a holomorphic section of $[\m^\cn]\cap\overline{\O_{X_4}}$, we can use the standard procedure of filling out zeros \cite[Proposition 2.2]{burstall-wood}   at points where  $(A^\varphi_z)^2$ does not have maximal rank in order to make $A=\im(A^\varphi_z)^2$ a $D^\varphi_z$-holomorphic subbundle of $\varphi^\perp$. As above, denote by $A_a$ the annihilator of $A$ and set
$$
\D=A_a^2\cap A_a\cap\varphi.
$$
It follows easily that $\D\cap \bar\D^\perp$ is a holomorphic,  isotropic subbundle of $\varphi$ contained in the kernel of $A^\varphi_z$. Hence, any holomorphic line subbundle of $\D\cap \bar\D^\perp$ will suffice as a twistor lift of $\varphi$. What remains is therefore only to show that there exists a line subbundle of $\D\cap \bar\D^\perp$. If $M$ is non-compact and all holomorphic bundles therefore holomorphically trivial, finding a holomorphic line subbundle of $\D\cap \bar\D^\perp$ is certainly possible. When $M$ is compact it is well-known that there exists a non-trivial meromorphic section $\eta$ of $\D\cap \bar\D^\perp$. The zeros and poles of $\eta$ constitute a finite set, and we can use the standard method to ``extend" the line bundle $\spa\{\eta\}$ across this set. Hence we have a holomorphic line subbundle of $\D\cap \bar\D^\perp$ across the entire surface. This completes the proof.

%
%
%
%

\subsubsection{The case $X=X_3+X_4$.}\label{XX34}

In this case, $X$  acts as follows with respect to the fundamental representation:
\begin{align*}
   & W_{-12}\mapsto W_{-11} \mapsto W_{-10}\mapsto 0,\quad    W_{10}\mapsto W_{11} \mapsto W_{12}\mapsto 0\\ &  W_{-9}\mapsto W_{-7} \mapsto W_{-3} \mapsto 0,\quad  W_{3}\mapsto W_{7} \mapsto W_{9}\mapsto 0  \\
   &   W_{-8}\mapsto W_{-6} \mapsto
     W_{-2}\mapsto 0,\quad   W_{2}\mapsto W_{6} \mapsto W_{8}\mapsto 0 \\
     & W_{-5}\mapsto W_{-1}+W_{-4}  \mapsto W_{0} \mapsto W_{1}+W_{4} \mapsto W_{5}\mapsto 0
  \end{align*}

\begin{lemma}$\im X^3_{|\varphi^\perp}$ is a holomorphic line subbundle of $\varphi$.
\end{lemma}

\begin{proof} It is easy to see that the root spaces associated to the roots $\alpha_3$, $\alpha_4$ and $\alpha_3+\alpha_4$ and their commutators generate a copy of $\sl_3$ inside $\f_4$, the action of which on
$$
W_5 + W_1 + W_4 + W_0 + W_{-1} + W_{-4} + W_{-5}
$$
is precisely its adjoint representation. From a simple calculation we thus see that $X^3(W_{-1} + W_{-4})=W_{5}$.
\end{proof}

Since $\ell=\im X^3_{|\varphi^\perp}$ is easily seen to be in the kernel of $A^\varphi_z$, this defines a $J_2$-holomorphic twistor lift of $\varphi$ into $T_4$.

This concludes the proof of Theorem \ref{thm-twistor}.

\begin{remark} Recall that
$$
(\f_4)^\cn=(\f_4)^{H_4}_{-2}\oplus(\f_4)^{H_4}_{-1}\oplus(\f_4)^{H_4}_0\oplus(\f_4)^{H_4}_1\oplus(\f_4)^{H_4}_2.
$$
We see from this that the twistor space $T_4$ is a  3-symmetric space and at the base point $x_0=eK$, with $K=\Spin(7)\times \U(1) $ we have
\begin{equation}\label{TJ2}
T^{1,0}_{J_2}T_4=(\f_4)^{H_4}_1\oplus(\f_4)^{H_4}_{-2}.
\end{equation}
In this case, we see from \eqref{gj} and \eqref{TJ2} that  $(\f_4)^1_{H_4}$  coincides with  $T^{1,0}_{J_2}T_4$,  which means that a smooth map $\ell:M\to T_4$ is $J_2$-holomorphic if and only if $\ell$ is a \emph{primitive} map into $T_4$ \cite{BurstallPedit}. In particular, $\ell$ is harmonic with respect to a suitable metric of $T_4$.
Since $\h^0_3(\Oc)$ is $26$-dimensional, $T_4$ can be naturally embedded in  $\cn P^{25}$. This embedding is not  totally geodesic, and therefore, in general, will not preserve harmonicity.
\end{remark}

\section{Harmonic maps of finite uniton number in $\Oc P^2$.}\label{se:unitons}

As first observed by K. Uhlenbeck \cite{uhlenbeck}, equations \eqref{eq:integrability} and \eqref{eq:harmonicity} can be reformulated in terms of the flatness of the one-parameter family of connections $\d+A_\lambda^\varphi$ on $M\times \cn^n$, where
\begin{equation*}\label{Blambda}
A_\lambda^\varphi=(1-\lambda^{-1})A^\varphi_z\d z+(1-\lambda)A^\varphi_{\bar{z}}\d\bar{z}\qquad(\lambda\in S^1).
\end{equation*}
When $M$ is simply connected we may integrate $A^\varphi_\lambda$ to obtain an \emph{extended solution}
$$
\Phi:M\to \Omega G=\{\gamma: S^1\to G\;(\mbox{smooth}):\, \gamma(1)=e\},
$$
satisfying $\Phi^{-1}\d\Phi=A^\varphi_\lambda$ and $\Phi_{\lambda = -1}=\varphi$; such a map $\Phi$ is unique up to left multiplication by a constant loop.

When the Fourier series in $\lambda\in S^1$ associated to an extended solution has finitely many terms, the extended solution and the corresponding harmonic map are said to have \emph{finite uniton number}. It is well known that any harmonic map from the $2$-sphere into a compact Lie group has finite uniton number \cite{uhlenbeck}. Among the extended solutions of finite uniton number, the simplest case occurs when $\Phi:M\to \Omega G$ takes values in a $G$-conjugacy class of homomorphisms  $S^1\to G$. Such extended solutions  are said to be \emph{$S^1$-invariant}. Next we shall  establish the general form for  $S^1$-invariant extended solutions  corresponding to maps into  $\Oc P^2$.

We denote by $Gr(G)$ the \emph{Grassmannian model} \cite{pressley-segal} for the loop group $\Omega G$. When $G=\U(n)$, then we simply write $Gr = Gr(\U(n))$. This model associates to each loop $\gamma\in\Omega G$ the closed subspace  $V\in Gr(G)$ of  $\H=L^2(S^1,\cn^n)$ defined by $V=\gamma \H_+$, where  $\H_+$ is the closed subspace of  $\H$ consisting of Fourier series whose negative coefficients vanish. For example, from \cite[Proposition 8.5.1]{pressley-segal} we know that $Gr(\SO (n))=\big\{{V\in {Gr}:\,\bar{V}^{\perp}=\lambda V \big\}}$.

We can use the complex bilinear  product in $\h_3^0(\Oc)^\cn$ to define a product on the Hilbert space $\H$ of square-summable $\cn^{26}\cong \h_3^0(\Oc)^\cn$-valued functions on the circle: if $f,g\in \H$, then $(f\cdot g)(\lambda)=f(\lambda)\cdot g(\lambda)$. The Grassmannian model of $\Omega F_4$  is given by the following proposition, whose proof we omit since it is analogous to that of
Proposition 3.2 in \cite{CP2012} for the $G_2$ case (see also the proof of Theorem 8.6.2 in \cite{pressley-segal}).

\begin{proposition}
With respect to the fundamental representation of $F_4$, we have:
$$
{Gr}(F_4)=\{V\in {Gr}\big(\SO(26)\big):\,V^{sm}\cdot V^{sm}\subseteq V^{sm}\}.
$$
where $V^{sm}$ denotes the
subspace  of smooth functions in $V$, which is dense in V \cite{pressley-segal}.
\end{proposition}

As Segal \cite{segal} has  observed, a smooth map $\Phi:M\to \Omega G$ is an extended solution if and only if $W=\Phi\H_+$ satisfies
$\partial_z W  \subseteq  \lambda^{-1}W$ (the \emph{pseudo-horizontality} condition) and $\partial_{\bar z} W \subseteq W$ ($W$ is a holomorphic vector subbundle of $M\times \H$), with respect to any local complex coordinate system  $(U,z)$. Hence, given an  $S^1$-invariant extended solution $\Phi$, we have
$$
W=A_{s}\lambda^{s}+A_{s+1}\lambda^{s+1}+\ldots +\lambda^{r}A_{r}+\lambda^r\H_+,
$$
for some integers $s\leq r$, where $A_{s}\subseteq  A_{s+1}\subseteq \ldots \subseteq A_r$ is a \emph{superhorizontal} sequence of holomorphic subbundles of $M\times\cn^n$, i.e., the holomorphic subbundles $A_i$ satisfy $\partial_z A_i\subseteq A_{i+1}$.

Burstall and Guest \cite{B-G} have shown that, after a normalization procedure, if $G$ has trivial centre, any $S^1$-invariant extended solution takes values in the $G$-conjugacy class of a homomorphism
$$
\gamma_I(\lambda)=\exp{(-\i\ln(\lambda)\xi_I)},
$$
with $\xi_I=\sum_{i\in I}H_i$ as defined in Section \S \ref{sec:twistor}. Clearly, there are only three such conjugacy classes associated to  harmonic maps into $\Oc P^2$:  $I=\{3\}$, $I=\{4\}$, and $I=\{3,4\}$. The representation of the corresponding canonical elements $\xi_I$ have been described in Section \S \ref{Sec:canonicalprojections}. Hence we have:

\begin{theorem}Let $\varphi:M\to \Oc P^2$ be a  harmonic map associated to an $S^1$-invariant extended solution. Then $\varphi$ admits an extended solution $\Phi$ such that $W=\Phi \H_+:M\to Gr(F_4)$ is given by one of the following forms:
\begin{enumerate}
\item[(i)]
$$
W=\ell\lambda^{-2}+\bar\ell_s^\perp\lambda^{-1}+\ell_s+\bar \ell^\perp\lambda  +\lambda^2\H_+,
$$
where $\ell$ is a holomorphic subbundle of isotropic lines in $\h_3^0(\Oc)^\cn$ and $\ell_s$ is the stabilizer of $\ell$. In this case, $\psi=\bar\ell: M\to T_4$ is a $J_2$-holomorphic lift of $\varphi$.
\item[(ii)]
$$
W=\C\lambda^{-3}+\C_a^2\lambda^{-2}+{\bar{\C}}_s^\perp\lambda^{-1}+\C_s+{\bar{\C_a^2}}^\perp \lambda +{\bar{\C}}^\perp\lambda^2+\lambda^3\H_+,
$$
where $\C$ is holomorphic subbundle of isotropic two-planes in $\h_3^0(\Oc)^\cn$ satisfying  $\C^2=0$, $\C_s$ its stabilizer and $\C_a$ its annihilator. In this case, $\psi=\bar\C: M\to T_3$ is a $J_2$-holomorphic lift of $\varphi$.
\item[(iii)]
\begin{align*}
W=\ell\lambda^{-5}+\C\lambda^{-4}+&\C_\ell^3\lambda^{-3}+\C_\ell^2\lambda^{-2}+\C_\ell\lambda^{-1}\\&+{\bar\C}_\ell^\perp+{\bar{\C_\ell^2}}^\perp \lambda+{\bar{\C_\ell^3}}^\perp \lambda^2+\bar{\C}^\perp \lambda^3+{\bar\ell}^\perp\lambda^4+\lambda^5 \H_+,
\end{align*}
where $\ell\subset \mathcal{C}$, $\ell$ is a holomorphic subbundle of isotropic lines in $\C$, $\C$ is holomorphic subbundle of isotropic two-planes in  $\C$ satisfying  $\C^2=0$, and $\C_\ell$ is the maximal subbundle satisfying $\mathcal{C}_\ell\cdot \mathcal{C}\subseteq \ell$. In this case, $\psi=(\bar\ell,\bar\C): M\to T_{34}$ is a $J_2$-holomorphic lift of $\varphi$.
\end{enumerate}
\end{theorem}

\begin{example}\label{examplefu} Following the procedure introduced by  Burstall and Guest \cite{B-G} for obtaining harmonic maps of finite uniton number from a Riemann surface $M$ into an inner symmetric space in terms of meromorphic functions on $M$, we give next an explicit example of a $J_2$-holomorphic lift $\psi=\bar \ell$ into $T_4$ of a  harmonic map $\varphi$ from $S^2=\cn \cup\{\infty\}$ into $\Oc P^2$ associated to a $S^1$-invariant extended solution in the conjugacy class  $I=\{4\}$.

With the same notations of \cite{B-G}, since $\max\{\alpha(H_4)/\i :\, \alpha\in\Phi^+\}=2$, any such harmonic map $\varphi$ admits an extended solution of the form
$$
\Phi=\exp (C)\cdot\gamma:S^2\to F_4,
$$
where $\gamma(\lambda)=\exp(-\i\ln(\lambda) H_4)$ and
$$
C=c_0^1+c_0^2:S^2\to(\f_4)_1^{H_4}\oplus(\f_4)_2^{H_4}
$$
is a meromorphic function  satisfying
$$
\left(c_0^2\right)_z-\frac{1}{2!}\left[c_0^1,\left(c_0^1\right)_z\right]=0.
$$
Here $c_0^i$ is the component of $C$ in $(\f_4)_i^{H_4}$, with $i=1,2$. The corresponding holomorphic line bundle $\ell$ is then given by
\begin{equation}\label{ell4}
\ell=\rho(\exp\left(C\right))W_{-12}=\left(Id+\rho(C)+\frac{\rho(C)^2}{2}+\frac{\rho(C)^3}{3!}+\frac{\rho(C)^4}{4!}\right)W_{-12}.
\end{equation}


For each $\alpha$ with $\alpha(H_4)=\i$,  fix $X_\alpha\in(\f_4)_\alpha\subset(\f_4)_1^{H_4}$. Take
$$
c_0^1=z^2X_{\alpha_4}+zX_{\alpha_2+\alpha_3+\alpha_4},$$ then $$\left(c_0^1\right)_z=2zX_{\alpha_4}+X_{\alpha_2+\alpha_3+\alpha_4}.
$$
Since $[(\f_4)_{\alpha_4}, (\f_4)_{\alpha_2+\alpha_3+\alpha_4}]=0,$ we can take $c_0^2=0$. Fix vectors $w_i\in \h_3^0(\Oc)^\cn$ such that $W_i=\spa\{w_i\}$ and
$$
w_{-11}=X_{\alpha_4}w_{-12},\quad w_{-9}=X_{\alpha_2+\alpha_3+\alpha_4}w_{-12}.
$$
Since $\rho(C)^2=0$, from \eqref{ell4} we obtain
$$
\ell(z)=\langle w_{-12}+z^2w_{-11}+zw_{-9}\rangle.$$
\end{example}

\appendix

\section{$F_4$: roots and fundamental representation}\label{ap:roots}

In this appendix we describe the fundamental representation of $F_4$ in more Lie theoretic terms, making it possible to connect it to the general twistor theory for harmonic maps into Lie groups and (inner) symmetric spaces. Our main references are \cite{adams}, \cite{Ha} and \cite{Hu}.

Denote by $\f_4$ the Lie algebra of $F_4$ and let $\Phi^+=\{\alpha_1,\alpha_2,\alpha_3,\alpha_4\}$ be a set of positive simple roots, forming the following familiar Dynkin diagram:
\begin{center}
\begin{tikzpicture}
\draw[fill=black]
(-1.5, 0) circle [radius = .08] node [anchor = north] {$\alpha_1$}
(-0.5, 0) circle [radius = .08] node [anchor = north] {$\alpha_2$}
(0.5, 0) circle [radius = .08] node [anchor = north] {$\alpha_3$}
(1.5, 0) circle [radius = .08] node [anchor = north] {$\alpha_4$};
\draw
(-1.45, 0) -- (-0.55, 0)
(-0.45, -0.04) -- (0.45, -0.04)
(-0.45, 0.04)  -- (0.45, 0.04)
(0.55, 0) -- (1.45, 0);
\draw
(0,0) --++ (60:-.2)
(0,0) --++ (-60:-.2);
\end{tikzpicture}
\end{center}
The longest root of $\f_4$ is $2\alpha_1+3\alpha_2+4\alpha_3+2\alpha_4$, which is easily seen to be a fundamental dominant weight for $\f_4$ corresponding to the node $\alpha_1$ in the Dynkin diagram. By including the root $\alpha_0 = -(2\alpha_1+3\alpha_2+4\alpha_3+2\alpha_4)$ we get the extended Dynkin Diagram
\begin{center}
\begin{tikzpicture}
\draw[fill=black]
(-2.5, 0) circle [radius = .08] node [anchor = north] {$\alpha_0$}
(-1.5, 0) circle [radius = .08] node [anchor = north] {$\alpha_1$}
(-0.5, 0) circle [radius = .08] node [anchor = north] {$\alpha_2$}
(0.5, 0) circle [radius = .08] node [anchor = north] {$\alpha_3$}
(1.5, 0) circle [radius = .08] node [anchor = north] {$\alpha_4$};
\draw
(-2.45, 0) -- (-1.55, 0)
(-1.45, 0) -- (-0.55, 0)
(-0.45, -0.04) -- (0.45, -0.04)
(-0.45, 0.04)  -- (0.45, 0.04)
(0.55, 0) -- (1.45, 0);
\draw
(0,0) --++ (60:-.2)
(0,0) --++ (-60:-.2);
\end{tikzpicture}
\end{center}
By removing $\alpha_4$ from this we recover the Dynkin diagram of the subalgebra $\spin(9)$ in $\f_4$. This also shows that, as representations of $\spin(9)$ we have
$$
\f_4 = \spin(9) + \Delta_9,
$$
where $\Delta_9$ is the spin-representation of $\spin(9)$. It follwos that that the positive roots of $\f_4$ are $\Phi^+$ and
\begin{align*}
&\alpha_1+\alpha_2,\ \alpha_2+\alpha_3,\ \alpha_3+\alpha_4,\ \alpha_2+2\alpha_3,\ \alpha_1+\alpha_2+\alpha_3,\ \alpha_2+\alpha_3+\alpha_4,\\
& \alpha_1+\alpha_2+2\alpha_3,\ \alpha_1+\alpha_2+\alpha_3+\alpha_4,\ \alpha_2+2\alpha_3+\alpha_4,\ \alpha_1+2\alpha_2+2\alpha_3,\\
&\alpha_2+2\alpha_3+2\alpha_4,\ \alpha_1+\alpha_2+2\alpha_3+\alpha_4,\ \alpha_1+\alpha_2+2\alpha_3+2\alpha_4,\\
&\alpha_1+2\alpha_2+2\alpha_3+\alpha_4,\ \alpha_1+2\alpha_2+2\alpha_3+2\alpha_4,\ \alpha_1+2\alpha_2+3\alpha_3+\alpha_4,\\
&\alpha_1+2\alpha_2+3\alpha_3+2\alpha_4,\ \alpha_1+2\alpha_2+4\alpha_3+2\alpha_4,\ \alpha_1+3\alpha_2+4\alpha_3+2\alpha_4,\\
&2\alpha_1+3\alpha_2+4\alpha_3+2\alpha_4.
\end{align*}
The roots with coefficient $1$ in front of $\alpha_4$ are weights of $\spin(9)$ on $\Delta_9$, the rest are the roots of $\spin(9)$.

Since left multiplication with $P=\mathrm{diag}(1,0,0)$ on $\h_3(\Oc)$ commutes with the action of $\spin(9)$, the decomposition of $\h_3(\Oc)$ in \eqref{dec} is also a decomposition of $\h_3(\Oc)$ into $\spin(9)$ representations. In fact we have
$$
A_{\frac{1}{2}}(P)=\Delta_9
$$
and
$$
A_0(P)=1 + \lambda_9^1,
$$
where $1$ denotes the trivial representation spanned by $\mathrm{diag}(0,1,1)$ and $\lambda_9^1$ the vector representation of $\spin(9)$. As $\spin(9)$ acts trivially on $A_1(P)$ we get
$$
\h_3(\Oc) = 2\cdot 1 + \Delta_9 + \lambda_9^1.
$$
Taking the trace free part we get the fundamental representation of $F_4$ decomposing under $\spin(9)$ as
$$
\h_3^0(\Oc) = 1 + \Delta_9 + \lambda_9^1,
$$
where the trivial representation is spanned by $\mathrm{diag}(-2, 1, 1)$. With this we can now calculate the weights of the fundamental representation and find that these are precisely the short roots of $\f_4$:
\begin{align*}
&w_{12}=\alpha_1+2\alpha_2+3\alpha_3+2\alpha_4,\ w_{11}=\alpha_1+2\alpha_2+3\alpha_3+\alpha_4,\\
&w_{10}=\alpha_1+2\alpha_2+2\alpha_3+\alpha_4,\ w_9=\alpha_1+\alpha_2+2\alpha_3+\alpha_4,\\
&w_8=\alpha_2+2\alpha_3+\alpha_4,\ w_7=\alpha_1+\alpha_2+\alpha_3+\alpha_4,\ w_6=\alpha_2+\alpha_3+\alpha_4,\\
&w_5=\alpha_3+\alpha_4,\ w_4=\alpha_4,\ w_3=\alpha_1+\alpha_2+\alpha_3,\ w_2=\alpha_2+\alpha_3,\,\,  w_1=\alpha_3,
\end{align*}
plus their negatives, all of which have multiplicity $1$, and $w_0=0$ with multiplicity $2$. We will henceforth denote by $W_i$ the weight space associated to the weights $w_i$. Note that, for any two weight spaces $W_i$ and $W_j$, we have
$$
W_i\cdot W_j \subseteq\text{the weight space associated to the weight $w_i + w_j$,}
$$
where the corresponding weight space is zero in case $w_i + w_j$ is not a weight. The full multiplication table for the weight spaces is as follows:
%

\begin{table}[!htbp]
\resizebox{\textwidth}{!}{%
\begin{tabular}[scale = 0.5]{c||c|c|c|c|c|c|c|c|c|c|c|c|c} $\cdot$ & $W_0$ & $W_1$ & $W_2$ & $W_3$ & $W_4$ & $W_5$ & $W_6$ & $W_7$ & $W_8$ & $W_9$ & $W_{10}$ & $W_{11}$ & $W_{12}$\\ \hline\hline $W_{-12}$ & $W_{-12}$ & 0 & 0 & 0 & $W_{-11}$ & $W_{-10}$ & $W_{-9}$ & $W_{-8}$ & $W_{-7}$ & $W_{-6}$ & $W_{-5}$ & $W_{-4}$ & $W_0$\\ \hline $W_{-11}$ & $W_{-11}$ & $W_{-10}$ & $W_{-9}$ & $W_{-8}$ & 0 & 0 & 0 & 0 & $W_{-3}$ & $W_{-2}$ & $W_{-1}$ & $W_0$ & $W_4$\\ \hline $W_{-10}$ & $W_{-10}$ & 0 & $W_{-7}$ & $W_{-6}$ & 0 & 0 & $W_{-3}$ & $W_{-2}$ & 0 & 0 & $W_0$ & $W_1$ & $W_5$\\ \hline $W_{-9}$ & $W_{-9}$ & $W_{-7}$ & 0 & $W_{-5}$ & 0 & $W_{-3}$ & 0 & $W_{-1}$ & 0 & $W_0$ & 0 & $W_2$ & $W_6$\\ \hline $W_{-8}$ & $W_{-8}$ & $W_{-6}$ & $W_{-5}$ & 0 & 0 & $W_{-2}$ & $W_{-1}$ & 0 & $W_0$ & 0 & 0 & $W_3$ & $W_7$\\ \hline $W_{-7}$ & $W_{-7}$ & 0 & 0 & $W_{-4}$ & $W_{-3}$ & 0 & 0 & $W_0$ & 0 & $W_1$ & $W_2$ & 0 & $W_8$\\ \hline $W_{-6}$ & $W_{-6}$ & 0 & $W_{-4}$ & 0 & $W_{-2}$ & 0 & $W_0$ & 0 & $W_1$ & 0 & $W_3$ & 0 & $W_9$\\ \hline $W_{-5}$ & $W_{-5}$ & $W_{-4}$ & 0 & 0 & $W_{-1}$ & $W_0$ & 0 & 0 & $W_2$ & $W_3$ & 0 & 0 & $W_{10}$\\ \hline $W_{-4}$ & $W_{-4}$ & 0 & 0 & 0 & $W_0$ & $W_1$ & $W_2$ & $W_3$ & 0 & 0 & 0 & 0 & $W_{11}$\\ \hline $W_{-3}$ & $W_{-3}$ & 0 & 0 & $W_0$ & 0 & 0 & 0 & $W_4$ & 0 & $W_5$ & $W_6$ & $W_8$ & 0\\ \hline $W_{-2}$ & $W_{-2}$ & 0 & $W_0$ & 0 & 0 & 0 & $W_4$ & 0 & $W_5$ & 0 & $W_7$ & $W_9$ & 0\\ \hline $W_{-1}$ & $W_{-1}$ & $W_0$ & 0 & 0 & 0 & $W_4$ & 0 & 0 & $W_6$ & $W_7$ & 0 & $W_{10}$ & 0\\ \hline $W_0$ & $W_0$ & $W_1$ & $W_2$ & $W_3$ & $W_4$ & $W_5$ & $W_6$ & $W_7$ & $W_8$ & $W_9$ & $W_{10}$ & $W_{11}$ & $W_{12}$\\ \hline $W_1$ & $W_1$ & 0 & 0 & 0 & $W_5$ & 0 & $W_8$ & $W_9$ & 0 & 0 & $W_{11}$ & 0 & 0\\ \hline $W_2$ & $W_2$ & 0 & 0 & 0 & $W_6$ & $W_8$ & 0 & $W_{10}$ & 0 & $W_{11}$ & 0 & 0 & 0\\ \hline $W_3$ & $W_3$ & 0 & 0 & 0 & $W_7$ & $W_9$ & $W_{10}$ & 0 & $W_{11}$ & 0 & 0 & 0 & 0\\ \hline $W_4$ & $W_4$ & $W_5$ & $W_6$ & $W_7$ & 0 & 0 & 0 & 0 & 0 & 0 & 0 & $W_{12}$ & 0\\ \hline $W_5$ & $W_5$ & 0 & $W_8$ & $W_9$ & 0 & 0 & 0 & 0 & 0 & 0 & $W_{12}$ & 0 & 0\\ \hline $W_6$ & $W_6$ & $W_8$ & 0 & $W_{10}$ & 0 & 0 & 0 & 0 & 0 & $W_{12}$ & 0 & 0 & 0\\ \hline $W_7$ & $W_7$ & $W_9$ & $W_{10}$ & 0 & 0 & 0 & 0 & 0 & $W_{12}$ & 0 & 0 & 0 & 0\\ \hline $W_8$ & $W_8$ & 0 & 0 & $W_{11}$ & 0 & 0 & 0 & $W_{12}$ & 0 & 0 & 0 & 0 & 0\\ \hline $W_9$ & $W_9$ & 0 & $W_{11}$ & 0 & 0 & 0 & $W_{12}$ & 0 & 0 & 0 & 0 & 0 & 0\\ \hline $W_{10}$ & $W_{10}$ & $W_{11}$ & 0 & 0 & 0 & $W_{12}$ & 0 & 0 & 0 & 0 & 0 & 0 & 0\\ \hline $W_{11}$ & $W_{11}$ & 0 & 0 & 0 & $W_{12}$ & 0 & 0 & 0 & 0 & 0 & 0 & 0 & 0\\ \hline $W_{12}$ & $W_{12}$ & 0 & 0 & 0 & 0 & 0 & 0 & 0 & 0 & 0 & 0 & 0 & 0 \end{tabular}
 }
  \caption{Multiplication table for the weight spaces}\label{table_multiplication_weights}
 \end{table}

%

\section{Nilpotent orbits associated to a symmetric space}\label{ap:nilpotent}

In this appendix we give a brief outline of the classification of nilpotent orbits associated to a symmetric space $G/K$, where $G$ is a semisimple compact Lie group. We give only the details needed to justify our results; for a more thorough exposition, including proofs, the reader is referred to \cite{Collingwood}.

Let $\g$ be a semisimple compact Lie algebra. Given a non-zero nilpotent element $X\in\g^\cn$, we say that the triple $\{X, Y, H\}$ is a \emph{standard triple} for $X$ if $Y, H\in\g^\cn$ satisfy
$$
[H, X]=2X,\quad [H,Y]=-2Y,\quad [X,Y]=H.
$$
The \emph{Jacobson-Morozov} theorem states that standard triples for a non-zero nilpotent element always exists. As the triple is a copy of an $\sl_2$ subalgebra in $\g^\cn$, the adjoint action of $H$ on $\g^\cn$ will have integer eigenvalues and decompose $\g^\cn$ into eigenspaces. Denote by $\g^H_i$ the eigenspace of $H$ with eigenvalue $i\in\zn$. From this, we may form the \emph{Jacobson-Morozov parabolic subalgebra} associated to $X$:
$$
\p=\l\oplus\u,
$$
where
$$
\l=\g_0^H\ \text{and}\ \u=\sum_{i>0}\g^H_i.
$$
The nilpotent element $X\in\g^\cn$ is said to be \emph{distinguished} (in $\g^\cn$) if the only Levi subalgebra of $\g^\cn$ containing $X$ is $\g^\cn$ itself. It turns out that $X$ is distinguished if and only if
$$
\di\l=\di\u/[\u,\u].
$$
Any parabolic subalgebra $\p=\l\oplus\u\subset\g^\cn$ with Levi factor $\l$ satisfying this equality is said to be \emph{distinguished} (in $\g^\cn$). It is known that a distinguished parabolic subalgebra is the Jacobson-Morozov parabolic subalgebra of a distinguished element, and this distinguished element lies in $[\u,\u]^\perp\cap\u=\g^H_2$.

Now, let $\tilde\l$ be a Levi subalgebra of $\g^\cn$ and $\tilde\p$ a distinguished parabolic subalgebra of the semisimple algebra $[\tilde\l,\tilde\l]$. The subalgebra $\tilde\p$ is the Jacobson-Morozov parabolic subalgebra of a distinguished element $X$ in $[\tilde\l,\tilde\l]$. The element $X$ is also nilpotent in $\g^\cn$. It can be shown that this correspondence between pairs $(\tilde\l,\tilde\p)$, with $\tilde\l$ a Levi subalgebra of $\g^\cn$ and $\tilde\p$ a distinguished parabolic subalgebra of $[\tilde\l,\tilde\l]$, is invariant under the adjoint action of $G^\cn$. Moreover, the induced correspondence between $G^\cn$-conjugacy classes of such pairs and nilpotent orbits in $\g^\cn$ is one-to-one. Thus, the classification of such pairs $(\tilde\l,\tilde\p)$ gives a classification of the nilpotent orbits in $\g^\cn$.

Assume now that $\g=\k\oplus \m$ is a symmetric decomposition of $\g$. J. Sekiguchi \cite{Sekiguchi} proved that there is a natural bijection between nilpotent $K^\cn$-orbits in $\m^\cn$ and nilpotent $G_\rn$-orbits in $\g_\rn$, where   $G_\rn$ is the noncompact real form of $G^\cn$ associated to the compact symmetric space $G/K$ and  $\g_\rn$ its Lie algebra. These real orbits were classified by D. Djokovi\'{c} \cite{DjokovicI,DjokovicII}.

We now apply this to the Cayley plane $\Oc P^2=F_4/\Spin(9)$ with the symmetric deomposition $\f_4=\spin(9)\oplus\m$ given by \eqref{symdecomposition} with $I=\{3,4\}$. The corresponding non-compact real form of $F_4^\cn$ is $F_{4(-20)}$ and, according to the classification of nilpotent $F_{4(-20)}$-orbits (see \cite{Collingwood}, p. 151) there are two non-trivial nilpotent $\Spin(9)^\cn$-orbits in $\m^\cn$. Comparing with the classification of nilpotent $F_4^\cn$-orbits in $\f_4^\cn$ (see \cite{Collingwood}, p 128) we can identify these two orbits as follows.
%
 \begin{enumerate}
 \item One of these orbits intersects the  nilpotent $F_4^\cn$-orbit that corresponds to the Levi subalgebra $\tilde{\l}=\tilde{A}_1\cong \sl(2)$ with a  short $\f_4$ simple root, say $\alpha_4$. Clearly,   the subalgebra $\tilde{\q}$ of $[\tilde{\l},\tilde{\l}]$  generated by $H_4$ and the root space $(\f_4)_{\alpha_4}$ is a distinguished subalgebra. Then, any nonzero element $X_4\in(\f_4)_{\alpha_4}\subset \m^\cn$  is a representative of this nilpotent orbit.
\item The second intersects the  nilpotent $F_4^\cn$-orbit that corresponds to the Levi subalgebra $\tilde{\l}=\tilde{A}_2\cong \sl(3)$ with the short $\f_4$ simple roots $\{\alpha_3,\alpha_4\}$. A distinguished parabolic subalgebra of $[\tilde{\l},\tilde{\l}]$ is the subalgebra $\tilde{\q}=\tilde{\l}\oplus \tilde{\u}$ generated by $H_3$, $H_4$ and the root spaces $(\f_4)_{\alpha_3}$, $(\f_4)_{\alpha_4}$ and $(\f_4)_{\alpha_3+\alpha_4}$. A representative $X$ of this nilpotent orbit lies in
$$
[\tilde{\u},\tilde{\u}]^\perp \cap \tilde{\u}=(\f_4)_{\alpha_3}\oplus (\f_4)_{\alpha_4}.
$$
Observe that $X$ must have nonzero components both in $(\f_4)_{\alpha_3}$ and $(\f_4)_{\alpha_4}$, otherwise $X$ would not be distinguished in $\tilde{A}_2$. Then $X$ is of the form $X=X_3+X_4$, with $X_3$ and $X_4$ nonzero elements of $(\f_4)_{\alpha_3}$ and $(\f_4)_{\alpha_4}$, respectively.
\end{enumerate}

\end{document}